\documentclass[12pt]{amsart}
\usepackage{amscd}      
\usepackage{amssymb}
\usepackage{amsmath, amsthm, graphics}
\allowdisplaybreaks[3]

\usepackage{latexsym}
\usepackage{epsfig}
\usepackage{amsfonts}
\usepackage{enumerate}
\usepackage{times}

\newtheorem{theorem}{Theorem}

\newtheorem{lemma}[theorem]{Lemma}
\newtheorem*{thm*}{Theorem}

\newtheorem{proposition}[theorem]{Proposition}
\theoremstyle{remark}
\newtheorem*{remark*}{Remark}

\def\<{\left\langle}
\def\>{\right\rangle}

\newcommand{\exterior}{\textstyle{\bigwedge}}

\begin{document}

\title{On the quantum $K$-ring of the flag manifold}

\date{\today}
\author[Anderson]{David Anderson}
\address{Department of Mathematics\\ Ohio State University\\ 100 Math Tower, 231 West 18th Ave. \\ Columbus,  OH 43210\\ USA}
\email{anderson.2804@math.osu.edu}

\author[Chen]{Linda Chen}
\address{Department of Mathematics and Statistics\\ Swarthmore College\\ Swarthmore, PA 19081\\ USA}
\email{lchen@swarthmore.edu}

\author[Tseng]{Hsian-Hua Tseng}
\address{Department of Mathematics\\ Ohio State University\\ 100 Math Tower, 231 West 18th Ave. \\ Columbus,  OH 43210\\ USA}
\email{hhtseng@math.ohio-state.edu}

\thanks{D. A. is supported in part by NSF grant DMS-1502201. L. C. is supported in part by Simons Foundation Collaboration grant 524354. H.-H. T. is supported in part by NSF grant DMS-1506551.}

\begin{abstract}
We establish a finiteness property of the quantum K-ring of the complete flag manifold.
\end{abstract}

\maketitle

The main aim of this note is to prove a fundamental fact about the quantum $K$-ring of the complete flag manifold.

\begin{thm*}
The structure constants for (small) quantum multiplication of Schubert classes in $QK_T(Fl_{r+1})$ are polynomials in the equivariant and Novikov variables.
\end{thm*}

\noindent
This is proved as Theorem~\ref{t.finite} below.  A priori, quantum structure constants are power series in the Novikov variables; our theorem says that in fact, only finitely many degrees appear.  This property is sometimes referred to as {\em finiteness} of the quantum product.

Finiteness is known for Grassmannians, and more generally for cominuscule homogeneous spaces \cite{bcmp1}.  On the other hand, there are conjectural ring presentations for $QK_T(Fl_{r+1})$ which presume finiteness and which also include a precise connection with the $K$-homology of the affine Grassmannian \cite{iim,kpsz,llms,lm}.

The proof of finiteness for cominuscule spaces relies on an understanding of the geometry of the moduli space of stable maps; in particular, certain subvarieties of the moduli space are shown to be rational \cite{bcmp2}.  Our proof, by contrast, is based on the reconstruction methods of Iritani-Milanov-Tonita (who also carried out computations for $Fl_3$) \cite{imt}; the argument consists of an analysis of the $K$-theoretic $J$-function.

\subsection*{Flag varieties}
Let $r\geq 1$ be an integer. Let $Fl_{r+1}$ be the variety of complete flags in $\mathbb{C}^{r+1}$. Let $$P_1,\ldots, P_r$$ be the pull-backs via the Pl\"ucker embedding $$Fl_{r+1}\to \prod_{i=1}^r \mathbb{P}^{{r+1 \choose i}-1}$$ of the line bundles $\mathcal{O}(-1)$ on the projective space factors.  Equivalently, if $S_1\subset \cdots \subset S_r \subset \mathbb{C}^{r+1}_{Fl}$ is the tautological flag of bundles on $Fl_{r+1}$, then $P_i = \exterior^i S_i$.

The {\em Novikov variables} $Q_1,\ldots, Q_r$ keep track of curve classes in $H_2(Fl_{r+1}, \mathbb{Z})$: to such a class $d$, we assign the monomial $\prod_{i=1}^rQ_i^{-\int_d c_1(P_i)}$. 

Let $\Lambda_1,\ldots, \Lambda_{r+1}$ be characters of the torus $T:=(\mathbb{C}^*)^{r+1}$ for the standard action on $\mathbb{C}^{r+1}$, inducing an action on $Fl_{r+1}$.  The $T$-equivariant $K$-ring of $Fl_{r+1}$ has the following well-known presentation:
\begin{equation}\label{eqn:presentation_classical}
K_T(Fl_{r+1})\simeq\frac{\mathbb{C}[\Lambda_1^\pm,\ldots, \Lambda_{r+1}^\pm; P_1^\pm,\ldots, P_{r}^\pm]}{\langle  H_k=e_k(\Lambda_1,\ldots, \Lambda_{r+1}) \rangle_{1\leq k\leq r+1}}. 
\end{equation}
Here $$H_k:=\sum_{I=\{i_1<\cdots<i_k\}\subset \{1,\ldots,r+1\}}\prod_{i\in I} P_i P_{i-1}^{-1}$$ (setting $P_0$ and $P_{r+1}$ equal to $1$) and $e_k$ is the $k$-th elementary symmetric polynomial in $r+1$ variables.  (See, e.g., \cite[Chapter IV, Section 3]{k}.)
 
\subsection*{Toda systems}
For $k\geq 1$, the $q$-difference Toda Hamiltonian is defined by
\begin{equation}
H_k^{q-\text{Toda}}(\mathfrak{p}_i,\mathfrak{z}_i):=\sum_{I=\{i_1<\cdots<i_k\}\subset \{1,\ldots,r+1\}}\prod_{l=1}^k\left(1-\frac{\mathfrak{z}_{i_l-1}}{\mathfrak{z}_{i_l}}\right)^{1-\delta_{i_l-i_{l-1},1}}\prod_{i\in I} \mathfrak{p}_i. 
\end{equation}
By convention, $i_0=0$. Note that $H_k^{q-\text{Toda}}$ depends on  $\mathfrak{z}_i$ through $\mathfrak{z}_{i-1}/\mathfrak{z}_i$. 

As an example, 
$$H_1^{q-\text{Toda}}=\mathfrak{p}_1+\sum_{i=2}^{r+1}\mathfrak{p}_i\left(1-\frac{\mathfrak{z}_{i-1}}{\mathfrak{z}_i} \right).$$
The above presentation is taken from \cite[Section 5.2]{kpsz}.  To obtain the $q$-difference Toda operators, as in \cite[Section 0]{gl}, we make the substitutions\footnote{We use the variables $t_1,\ldots, t_{r+1}$ while $t_0,\ldots, t_{r}$ are used in \cite{gl}.}
\begin{equation}
\mathfrak{p}_i\mapsto q^{\partial_{t_i}}, \quad \mathfrak{z}_i\mapsto e^{t_i},
\end{equation}
and identify $e^{t_i-t_{i+1}}$ with the Novikov variable $Q_i$ for $i=1,\ldots, r$.  Thus
\begin{equation}
\begin{split}
H_k^{q-\text{Toda}}(q^{\partial_{t_i}},e^{t_i})&=\sum_{I=\{i_1<\cdots<i_k\}\subset \{1,\ldots,r+1\}}\prod_{l=1}^k\left(1-\frac{e^{t_{i_l-1}}}{e^{t_{i_l}}}\right)^{1-\delta_{i_l-i_{l-1},1}}\prod_{i\in I} q^{\partial_{t_i}}\\
&=\sum_{I=\{i_1<\cdots<i_k\}\subset \{1,\ldots,r+1\}}\prod_{l=1}^k\left(1-Q_{i_l-1}\right)^{1-\delta_{i_l-i_{l-1},1}}\prod_{i\in I} q^{\partial_{t_i}}.
\end{split}
\end{equation}
Note that 
\begin{align*}
    \partial_{t_i}&=Q_i\partial_{Q_i}-Q_{i-1}\partial_{Q_{i-1}} \quad\text{ for } 2\leq i\leq r,\\
    \partial_{t_1}&=Q_1\partial_{Q_1}, \quad \text{and} \\ \quad \partial_{t_{r+1}}&=-Q_{r}\partial_{Q_{r}}.   
\end{align*}
Thus the Toda Hamiltonian is written in terms of {\em $q$-shift operators} $q^{Q_i\partial_{Q_i}}$ which act on power series by
\[
  q^{Q_i\partial_{Q_i}}f(Q_1,\ldots,Q_i,\ldots,Q_r) = f(Q_1,\ldots,qQ_i,\ldots,Q_r).
\]
Similarly, the negative $q$-shift operator $q^{-Q_i\partial_{Q_i}}$ acts by replacing the variable $Q_i$ with $q^{-1}Q_i$.

It is proven in \cite{gl} that the $K$-theoretic $J$-function of $Fl_{r+1}$ is an eigenfunction of the $q$-difference Toda system.  More precisely, write $J(Q, q)$
for the $K$-theoretic $J$-function of $Fl_{r+1}$, as defined in \cite[Section 2.2]{gl}. Then we have 

\begin{theorem}[\cite{gl}, Corollary 2]\label{thm:eigen}
For $1\leq k\leq r+1$,
$$H_k^{q-\text{Toda}}(q^{\partial_{t_i}}, e^{t_i})P^{\log Q/\log q}J(Q, q)=e_k(\Lambda_1,\ldots, \Lambda_{r+1})P^{\log Q/\log q}J(Q, q),$$
where $P^{\log Q/\log q}:=\prod_{i=1}^r P_i^{\log Q_i/\log q}$. 
\end{theorem}

\subsection*{Relations}
The (small) quantum $K$-ring of $Fl_{r+1}$ is additively defined to be $$QK_T(Fl_{r+1}):=K_T(Fl_{r+1})\otimes_\mathbb{C}\mathbb{C}[[Q_1,\ldots,Q_r]]$$
and is equipped with a quantum product $\star$, deforming the tensor product on $K_T(Fl_{r+1})$. The structure constants of $\star$ are defined in a rather involved way using $K$-theoretic Gromov-Witten invariants of $Fl_{r+1}$.  See \cite{g} and \cite{l} for details.

Theorem \ref{thm:eigen} yields  relations in the small $T$-equivariant quantum $K$-ring of $Fl_{r+1}$ as follows.  The theorem gives $q$-difference equations satisfied by $P^{\log Q/\log q}J(Q, q)$.  The operators $H_k^{q-\text{Toda}}(q^{\partial_{t_i}}, e^{t_i})$ contain negative $q$-shift operators $q^{-Q_i\partial_{Q_i}}$, but the difference equations in Theorem~\ref{thm:eigen} are equivalent to the difference equations 
 \begin{equation*}
\begin{split}
 \left(\prod_{i=1}^r q^{Q_i \partial_{Q_i}}\right)H_k^{q-\text{Toda}}(q^{ \partial_{t_i}}, e^{t_i})P^{\log Q/\log q}J(Q, q)\\
 =e_k(\Lambda_1,\ldots, \Lambda_{r+1})\left(\prod_{i=1}^r q^{Q_i \partial_{Q_i}}\right)P^{\log Q/\log q}J(Q, q),
 \end{split}
\end{equation*}
which do not contain negative $q$-shift operators.

In \cite[\S2]{imt}, certain operators $A_{i,\text{com}}$ are defined, acting as endomorphisms of the $\mathbb{C}[[Q]]$-module $QK_T(Fl_{r+1})$.  By \cite[Proposition~2.6 and Corollary~2.9]{imt}, these operators commute with one another and act as quantum multiplication by the class $A_{i,\text{com}}1$, where $1\in QK_T(Fl_{r+1})=K_T(Fl_{r+1})\otimes \mathbb{C}[[Q]]$ is the identity element of the ring.  If one replaces each $q^{Q_i\partial_{Q_i}}$ by $A_{i,\text{com}}$, the difference equations become relations
\begin{equation*}
\begin{split}
 \left(\prod_{i=1}^r \mathsf{A}_{i, \text{com}}\right)H_k^{q-\text{Toda}}(\mathsf{A}_{i, \text{com}}\mathsf{A}_{i-1,\text{com}}^{-1}, e^{t_i})1\\
 =e_k(\Lambda_1,\ldots, \Lambda_{r+1})\left(\prod_{i=1}^r \mathsf{A}_{i, \text{com}}\right)1
 \end{split}
\end{equation*}
in the ring $QK_T(Fl_{r+1})$ \cite[Proposition~2.12]{imt}.  (Here we set $A_{0,\text{com}}=A_{r+1,\text{com}}=\mathrm{Id}$.)

Applying the operator $\prod_{i=1}^r \mathsf{A}_{i, \text{com}}^{-1}$ to the above relation, we find
\begin{equation}\label{eqn:relation}
H_k^{q-\text{Toda}}((\mathsf{A}_{i, \text{com}}1)\star (\mathsf{A}_{i-1,\text{com}}^{-1}1), e^{t_i})=e_k(\Lambda_1,\ldots, \Lambda_{r+1}).
\end{equation}
In view of \cite[Proposition 2.10]{imt}, we have\footnote{Our notation agrees with that of \cite{gl}, but differs slightly from \cite{imt}, where $P_i$ and $P_i^{-1}$ are interchanged.} 
\begin{equation*}
    \mathsf{A}_{i, \text{com}}1=P_i \text{ mod } Q \quad \text{and}\quad  \mathsf{A}_{i, \text{com}}^{-1}1=P_i^{-1} \text{ mod } Q.
\end{equation*}
Modulo $Q_i$, therefore, \eqref{eqn:relation} produces a complete set of relations of $K_T(Fl_{r+1})$, as one sees by comparing with \eqref{eqn:presentation_classical}.  It follows that the relations \eqref{eqn:relation} define the ring $QK_T(Fl_{r+1})$.

\subsection*{The $D_q$-module}
Set $\tilde{J}:=P^{\log Q/\log q}J$. The $D_q$-module structure established in \cite{gt} and elaborated in \cite{imt} implies the following. Let $f(Q, x)\in\mathbb{C}[Q_1,\ldots,Q_r, x_1,\ldots,x_r]$, then $$f(Q, q^{Q_i\partial_{Q_i}})(1-q)\tilde{J}=\sum_\beta (1-q)\tilde{T}(f_\beta \Phi_\beta),$$
where on the right-hand side we have a finite sum with $f_\beta\in \mathbb{C}[1-q][[Q]]$ and $\Phi_\beta\in K_T(Fl_{r+1})$. Here $\tilde{T}=P^{\log Q/\log q} T$, and $T$ is the fundamental solution considered in \cite[Section 2.4]{imt}.  The right-hand side can be computed from the leading terms in $q\to \infty$ limit of the left-hand side $f(Q, q^{Q_i\partial_{Q_i}})(1-q)\tilde{J}$, namely the coefficients of $q^{\geq 0}$. Furthermore, this implies the following equation in $QK_T(Fl_{r+1})$: $$f(Q, \mathsf{A}_{i,\text{com}})\star 1=\sum_\beta f_\beta\big|_{q=1} \Phi_\beta.$$ 
See the proof of \cite[Lemma 3.3]{imt}.

Now let us write the $J$-function as a series $J=\sum_d Q^d J_d$, with $d=(d_1,\ldots,d_r)\in (\mathbb{Z}_{\geq 0})^r$ and $Q^d=\prod_{i=1}^rQ_i^{d_i}$. For each degree $d$, $J_d$ is a rational function in $q$ taking values in $K_T(Fl_{r+1})$. Therefore, as $q\to \infty$, we have
\[
 J_d\sim C_d(P)q^{f(d)},
\]
for some $C_d(P)\in K_T(Fl_{r+1})$ and some function $f(d)$.  Results of \cite[Section 4]{gl} include an estimate on the function $f(d)$.
\begin{lemma}[{\cite[Eq.~(7)]{gl}}]\label{lem:deg_bound}
We have $f(d)\leq -k_d$, where
\[
  k_d:=d_1+\cdots+d_r+\sum_{i=1}^{r+1}\frac{(d_i-d_{i-1})^2}{2}.
\]
Here by convention we set $d_0=d_{r+1}=0$.
\end{lemma}

For a class $\Phi\in K_T(Fl_{r+1})$, we expand the fundamental solution by writing $T(\Phi)=\sum_d Q^d T_{d,\Phi}$.  From the definition of $T$, the coefficient $T_{d,\Phi}$ encodes two-point Gromov-Witten invariants with one descendant insertion.  This is a rational function in $q$, vanishing at $q=+\infty$; more precisely, as $q\to +\infty$, we have the asymptotics $T_{d,\Phi}\sim L_{d,\Phi}q^{v_{d,\Phi}}$.

\begin{lemma}\label{lem:deg_bound_gen}
$v_{d,\Phi}\leq -k_d$, where $k_d$ is as in Lemma \ref{lem:deg_bound}. 
\end{lemma}
\begin{proof}
Coefficients of $T(\Phi)$ are obtained by $\mathbb{C}^*$-equivariant localization on the graph space with one marked point, $\overline{\mathcal{M}}_{0,1}(Fl_{r+1}\times \mathbb{P}^1, (d,1))$. 

Consider the maps $$\overline{\mathcal{M}}_{0,1}(Fl_{r+1}\times \mathbb{P}^1, (d,1))\overset{\mu\times ev}{\longrightarrow}\Pi_d\times (Fl_{r+1}\times\mathbb{P}^1)\overset{\tilde{\lambda}}{\longleftarrow}HQ_d\times \mathbb{P}^1.$$ See \cite[Section 0]{gl} for the definition of $\Pi_d$ and \cite[Section 3]{gl} for discussions on hyperquot schemes $HQ_d$. The map $\tilde{\lambda}$ is defined by $\tilde{\lambda}(x, y):=(\lambda(x), ev(x, y), y)$. The maps $\mu$ and $\lambda$ are defined analogously to those in \cite[Sections 2.1 and 3.1]{gl}.

The varieties $\overline{\mathcal{M}}_{0,1}(Fl_{r+1}\times \mathbb{P}^1, (d,1))$ and $HQ_d\times \mathbb{P}^1$ are smooth stacks of expected dimension $\text{dim}\,Fl_{r+1}+2d_1+\cdots+2d_r+1$, so their virtual structure sheaves coincide with their structure sheaves.  Coefficients of $T(\Phi)$ are therefore obtained by $\mathbb{C}^*$-equivariant localization applied to $(\mu\times ev)_*(\mathcal{O})\otimes (1\otimes \Phi\otimes 1)$. 

Also, $\mu\times ev$ and $\tilde{\lambda}$ are birational onto their common image, which has rational singularities.  Hence coefficients of $T(\Phi)$ are obtained by $\mathbb{C}^*$-equivariant localization applied to $\tilde{\lambda}_*(\mathcal{O}_{HQ_d\times \mathbb{P}^1})\otimes(1\otimes\Phi\otimes 1)$. The result then follows from the arguments of \cite[Section 4.2]{gl}.
\end{proof}

\subsection*{The operator $\mathsf{A}_{i,\text{com}}$}

By studying the action of $q$-shift operators on the $J$-function, we can identify the operator $A_{i,\text{com}} \in \mathrm{End}(QK_T(Fl_{r+1}))$.

\begin{lemma}\label{lem:oper_A}
The operator $\mathsf{A}_{i,\text{com}}$ is the operator of (small) quantum product by $P_i$.
\end{lemma}
\begin{proof}
By \cite[Proposition 2.10]{imt}, $\mathsf{A}_{i,\text{com}}$ is the operator of (small) quantum product by $P_i+\sum_{d}c_d Q^d$, so it suffices to show that $\mathsf{A}_{i,\text{com}}1=P_i$. To this end we consider $$q^{Q_i\partial_{Q_i}}(1-q)\tilde{J}.$$ Its $q^{\geq 0}$ coefficients can come from two places:
\begin{enumerate}
    \item 
    $d=0$: in this case the factor $P^{\log Q/\log q}$ of $\tilde{J}$ contributes $P_i$.
    
    \item $d\neq 0$: if such $d$ exists, the effect of the difference operator $q^{Q_i\partial_{Q_i}}$ is $q^{d_i}Q^dJ_d$. For this term to contribute to the $q^{\geq 0}$ coefficient, we must have $d_i+f(d)\geq 0$. By Lemma \ref{lem:deg_bound}, we have
    \begin{equation*}
    \begin{split}
        0\leq d_i+f(d)&\leq d_i-k_d\\
        &=-\sum_{j\neq i}d_j -\sum_{i=1}^{r+1}\frac{(d_i-d_{i-1})^2}{2}.
    \end{split}
    \end{equation*}
The last term is non-positive because $d_j\geq 0$ for all $j$. And it is equal to $0$ if and only if $d_j=0$ for $j\neq i$ and $d_j-d_{j-1}=0$ for all $j$. This implies $d_1=\cdots=d_r=0$, which is not the case. 
\end{enumerate}
Thus we find $$q^{Q_i\partial_{Q_i}}(1-q)\tilde{J}=(1-q)\tilde{T}P_i,$$ and hence $\mathsf{A}_{i, \text{com}}1=P_i$. 
\end{proof}

\begin{remark*}
The same argument shows that for {\em distinct} $i_1,\ldots,i_l\in \{1,\ldots,r\}$, we have $$\left(\prod_{k=1}^lq^{Q_{i_k}\partial_{Q_{i_k}}}\right)(1-q)\tilde{J}=(1-q)\tilde{T}\left(\prod_{k=1}^lP_{i_k}\right),$$ and hence $P_{i_1}\star \cdots\star P_{i_l}=\prod_{k=1}^lP_{i_k}$.  That is, for these elements, the quantum and classical product are the same.
\end{remark*}

\subsection*{Finiteness}

The main ingredient in our theorem is a finiteness statement for products of the line bundle classes $P_i$.

\begin{proposition}\label{prop:monomial}
The (small) quantum product $P_{i_1}\star \cdots\star P_{i_l}$ is a finite linear combination of elements of $K_T(Fl_{r+1})$ whose coefficients are polynomials in $Q_1,\ldots,Q_r$.
\end{proposition}
\begin{proof}
Again we consider the $q^{\geq 0}$ coefficients of $\prod_{k=1}^l q^{ Q_{i_k}\partial_{Q_{i_k}}}(1-q)\tilde{J}$. The $d=0$ term of $\tilde{J}$ gives $\prod_{k=1}^lP_{i_k}$. For a $d\neq 0$ term of $\tilde{J}$ to contribute to the $q^{\geq 0}$ coefficient, we must have $$\sum_{k=1}^l d_{i_k}+f(d)\geq 0.$$
Each such term contributes $C'_d(P)Q^d$ to the $q^{\geq 0}$ coefficients, where $C'_d(P)$ is a polynomial in the $P_i$'s.  We need to show that there are only finitely many such terms. 

If $\sum_{k=1}^l d_{i_k}+f(d)\geq 0$, then
\begin{equation*}
    \begin{split}
        0\leq \sum_{k=1}^l  d_{i_k}+f(d)&\leq \sum_{k=1}^l  d_{i_k}-k_d\\
        &=\left(\sum_{k=1}^l  d_{i_k}-\sum_{j=1}^rd_j\right) -\sum_{i=1}^{r+1}\frac{(d_i-d_{i-1})^2}{2}.
    \end{split}
    \end{equation*}
    
The quadratic form $\sum_{i=1}^{r+1}\frac{(d_i-d_{i-1})^2}{2}$ in $d_1,\ldots,d_r$ is positive definite. Indeed it is nonnegative because it is a sum of squares. Also, if $\sum_{i=1}^{r+1}\frac{(d_i-d_{i-1})^2}{2}=0$, then $d_i-d_{i-1}=0$ for all $i=1,\ldots, r+1$. Because $d_0=d_{r+1}=0$, we have $d_1=\cdots=d_r=0$. Therefore level sets of the function of $d_1,\ldots,d_r$ $$\left(\sum_{k=1}^l d_{i_k}-\sum_{j=1}^rd_j\right) -\sum_{i=1}^{r+1}\frac{(d_i-d_{i-1})^2}{2}$$ are ellipsoids. It follows that the set $$\left\{(d_1,\ldots,d_r)\Big| \left(\sum_{k=1}^l  d_{i_k}-\sum_{j=1}^rd_j\right) -\sum_{i=1}^{r+1}\frac{(d_i-d_{i-1})^2}{2}\geq 0 \right\}$$ is a bounded subset of $\mathbb{R}^r$, so it can contain at most finitely many $(d_1,\ldots,d_r)\in (\mathbb{Z}_{\geq 0})^r$.   

The (finitely many) $q^{\geq 0}$ terms of $\prod_{k=1}^l q^{ Q_{i_k}\partial_{Q_{i_k}}}(1-q)\tilde{J}$ can be ordered according to the exponents of $q$. We then use terms $$q^n Q^d\tilde{T}(\Phi), \quad n\in \mathbb{Z}_{\geq 0}, d\in (\mathbb{Z}_{\geq 0})^{r}, \Phi\in K_T(Fl_{r+1})$$
to inductively remove these $q^{\geq 0}$ terms. 

By Lemma \ref{lem:deg_bound_gen}, $q^n Q^d\tilde{T}(\Phi)$ has only finitely many $q^{\geq 0}$ terms, so the inductive process ends after finitely many steps. This means  we can find a finite sum $\sum_\beta (1-q)\tilde{T}(f_\beta \Phi_\beta)$, with $f_\beta\in \mathbb{C}[1-q][Q]$ and $\Phi_\beta\in K_T(Fl_{r+1})$, such that 
\begin{align}\label{e.claim}
\prod_{k=1}^l q^{ Q_{i_k}\partial_{Q_{i_k}}}(1-q)\tilde{J}-\sum_\beta (1-q)\tilde{T}(f_\beta \Phi_\beta)
\end{align}
vanishes at $q=+\infty$.  
The Proposition then follows from the fact that the expression of \eqref{e.claim} equals zero, which is proved in Lemma~\ref{l.claim} below. 
\end{proof}

\begin{lemma}\label{l.claim}
With notation as in \eqref{e.claim} above, we have 
\[
\prod_{k=1}^l q^{ Q_{i_k}\partial_{Q_{i_k}}}(1-q)\tilde{J}=\sum_\beta (1-q)\tilde{T}(f_\beta \Phi_\beta).
\]
\end{lemma}

\begin{proof}
We argue as in the proof of \cite[Lemma 3.3]{imt}. Write $$M:=(1-q)^{-1}\left(P^{\log Q/\log q}\right)^{-1}\left(\prod_{k=1}^l q^{ Q_{i_k}\partial_{Q_{i_k}}}(1-q)\tilde{J}-\sum_\beta (1-q)\tilde{T}(f_\beta \Phi_\beta) \right).$$ Expand $M$ as a series in $Q$, we get $M=\sum_d M_d Q^d$. Then we get $M_0=0$. $M_d$ has poles only at $q=$roots of unity, $M_d$ is regular at both $q=0$ and $q=+\infty$ and vanishes at $q=+\infty$. 

By \cite[Remark 2.11]{imt}, we can write $M$ as $M=TU$ with $T=\sum_d T_d Q^d$, $U=\sum_d U_d Q^d$. Then $T_0=Id$ and $U_0=0$. $T_d$ has only poles at $q=$roots of unity. $T_d$ is regular at $q=0, +\infty$, and vanishes at $q=+\infty$. Also, $U_d$ is a Laurent polynomials.

We want to show that $U_d=0$ for all $d>0$, by induction on $d$ with respect to a partial order of $d$ ample class. For $d$ we have $$M_d=T_d+U_d+ \sum_{\overset{d'+d''=d,}{d', d''\neq 0}}T_{d'}U_{d''}.$$ $M_d$ is known, so this equation and induction determine $T_d+U_d$. 

Since both $T_d+U_d$ and $T_d$ are regular at $q=0$, so is $U_d$. So the Laurent polynomial $U_d$ has no $q^{<0}$ terms. Since both $T_d+U_d$ and $T_d$ are regular at $q=+\infty$, so is $U_d$. So the Laurent polynomial $U_d$ has no $q^{>0}$ terms. Since $T_d+U_d$ and $T_d$ vanish at $q=+\infty$, we have $U_d\Big|_{q=+\infty}=0$. Hence $U_d=0$.
\end{proof}

Finally, we turn to our main theorem.  Let us write $R(T)=\mathbb{C}[\Lambda_1^\pm,\ldots, \Lambda_{r+1}^\pm]$ for the representation ring of the torus, and $R(T)[Q]=R(T)[Q_1,\ldots,Q_r]$ and $R(T)[[Q]]=R(T)[[Q_1,\ldots,Q_r]]$.  Then $K_T(Fl_{r+1})$ is a free $R(T)$-module and $QK_T(Fl_{r+1})$ is a free $R(T)[[Q]]$-module.  Fix an $R(T)$-basis $\{\sigma_w\}$ for $K_T(Fl_{r+1})$, so $\{\Phi_w = \sigma_w\otimes 1\}$ is an $R(T)[[Q]]$-basis for $QK_T(Fl_{r+1})$.

\begin{theorem}\label{t.finite}
The structure constants of $QK_T(Fl_{r+1})$ with respect to the basis $\{\Phi_w\}$ are polynomials: they lie in the subring $R(T)[Q]\subset R(T)[[Q]]$.
\end{theorem}

\begin{proof}
It is a basic fact that $K_T(Fl_{r+1})$ is generated by $P_1,\ldots,P_{r+1}$ as an $R(T)$-algebra; that is, there is a surjective homomorphism
\[
  R(T)[P_1,\ldots,P_{r}] \twoheadrightarrow K_T(Fl_{r+1}).
\]
In particular, each basis element $\sigma_w$ can be written as a polynomial in $P_i$'s with coefficients in $R(T)$.  (One way to see this is as follows.  The presentation \eqref{eqn:presentation_classical} establishes $K_T(Fl_{r+1})$ as a quotient of $R(T)[P_1^{\pm},\ldots,P_{r}^{\pm}]$, so it suffices to write $P_i^{-1}$ as a polynomial in $P_i$ with coefficients in $R(T)$.  For each $i$, one can find monomials $\omega_{ij}$ in the variables $\Lambda^{\pm}$ so that
\[
  \prod_{j=1}^{{r+1 \choose i}} (1-\omega_{ij} P_i) = 0  \quad \text{in} \quad K_T(Fl_{r+1}),
\]
and re-arranging this equation produces the desired expression.  Alternatively, one can use the equivariant Riemann-Roch isomorphism together with the fact that the equivariant cohomology ring of $Fl_{r+1}$ is generated by divisor classes.)

The assignment $P_{i_1} P_{i_2} \cdots P_{i_k}\mapsto P_{i_1}\star P_{i_2}\star \cdots \star P_{i_k}$ defines a ring homomorphism 
\begin{equation}\label{eqn:ring_hom1}
R(T)[P_1,\ldots, P_r; Q_1,\ldots,Q_{r}]\to QK_T(Fl_{r+1});
\end{equation}
let the kernel be $\mathsf{I}$.  The resulting embedding of rings
\[
 R(T)[P_1,\ldots, P_{r}; Q_1,\ldots,Q_{r}]/\mathsf{I}\hookrightarrow QK_T(Fl_{r+1})
\]
corresponds to the natural embedding of modules
\[
  K_T(Fl_{r+1})\otimes \mathbb{C}[Q_1,\ldots,Q_r] \hookrightarrow K_T(Fl_{r+1})\otimes \mathbb{C}[[Q_1,\ldots,Q_r]].
\]
Since each basis element $\sigma_w$ is a polynomial in $\Lambda^\pm$ and $P$. it follows that each basis element $\Phi_w=\sigma_w\otimes 1$ can be represented as a polynomial $G_w=G_w(\Lambda^{\pm},P,Q)$ in $R(T)[P_1,\ldots, P_{r}; Q_1,\ldots,Q_{r}]$.  The product of basis elements $\Phi_u\star\Phi_v$ is represented by $G_u\, G_v$, and by Proposition~\ref{prop:monomial}, this product is a finite linear combination of classes in $K_T(Fl_{r+1})$ with coefficients in $\mathbb{C}[Q_1,\ldots,Q_r]$.
\end{proof}

\bigskip
\noindent
{\it Acknowledgements}. 
We thank A. Givental and H. Iritani for discussions on $D_q$-modules and quantum $K$-theory.

\end{document}